\documentclass[11pt,reqno,a4]{amsart}

\usepackage{a4}
\usepackage{graphicx}
\usepackage{amsmath,amssymb,natbib,color,bbm,ifthen,enumerate}
\usepackage{hyperref}
\usepackage[latin1]{inputenc}

\newtheorem{theorem}{Theorem}
\newtheorem{lemma}{Lemma}
\newtheorem{definition}{Definition}

\theoremstyle{remark}
\newtheorem{remark}{Remark}

\theoremstyle{definition}

\newtheorem{assumption}{Assumption}

\newenvironment{enum_A}
  {%
  \setlength{\leftmargini}{4em}\begin{enumerate}}
  {\end{enumerate}}

\newenvironment{enum_W}
  {%
  \setlength{\leftmargini}{4em}\begin{enumerate}}
  {\end{enumerate}}

  {\end{enumerate}}

  {\end{enumerate}}

\def\maxscale{J}

\def\lowscale{L}

\def\upscale{U}

\def\Sclt{{\mathbf{\widehat{S}}}}

\newcommand\jmean[1]{\langle #1 \rangle}

\def\dj{u}
\def\dk{v}
\def\rmd{\mathrm{d}}
\def\rme{\mathrm{e}}
\def\rmi{\mathrm{i}}

\def\cl{\mathop{\stackrel{\mathcal{L}}{\longrightarrow}}}

\def\hd{\hat{d}} 

\def\dwt{W}

\def\indexset{\mathcal{I}}

\def\L{\mathrm{T}}

\def\be{\mathbf{e}}

\def\diffop{\mathbf{\Delta}}

\def\1{\mathbbm{1}}

\def\bX{\mathbf{X}}

\def\bx{\mathbf{x}}

\def\Nset{\mathbb{N}} 
\def\Zset{\mathbb{Z}}

\def\Rset{\mathbb{R}} 
\def\PE{\mathbb{E}} 
\def\PVar{\mathrm{Var}}
\def\PCov{\mathrm{Cov}}

\def\calN{\mathcal{N}}

\def\ie{\textit{i.e.} }

\def\argmin{\mathop{\mathrm{Argmin}}}
\def\prob{\mathbb{P}}

\newcommand{\eqdef}{\ensuremath{\stackrel{\mathrm{def}}{=}}}
\newcommand{\eqsp}{\;}

\newcommand{\AVvar}[3][]
{
\ifthenelse{\equal{#1}{}}{\mathbf{V}_{#3}(#2)}{\mathbf{V}_{#3}(#2,#1)}}
\newcommand{\AVvarJoint}[3][]
{
\ifthenelse{\equal{#1}{}}{\mathbf{\Lambda}_{#3}(#2)}{\mathbf{\Lambda}_{#3}(#2,#1)}}

\newcommand{\AsympVarWWE}[2][]
{\rho^2(#2)}

\newcommand{\AVvarInv}[3][]
{\ifthenelse{\equal{#1}{}}{\mathbf{V}^{-1}_{#3}(#2)}{\mathbf{V}^{-1}_{#3}(#2,#1)}}
\newcommand{\sigmaasymp}[2][]
{
\ifthenelse{\equal{#1}{}}{\sigma(#2)}{\sigma(#2,#1)}}

\def\vjsymb{\sigma}
\newcommand{\vj}[4][]{%
\ifthenelse{\equal{#1}{}}{\vjsymb^{#4}_{#2}}{\vjsymb^{#4}_{#2}(#3,#1)}}
\newcommand{\stdj}[3][]{%
\ifthenelse{\equal{#1}{}}{\vjsymb_{#2}}{\vjsymb_{#2}(#3,#1)}}

\newcommand{\hvj}[3][]{%
\ifthenelse{\equal{#1}{}}{\hat{\vjsymb}^2_{#2}}{\hat{\vjsymb}^2_{#2}(#1)}}
\newcommand{\hvjLIM}[3][]{%
\ifthenelse{\equal{#1}{}}{Q^{(#2)}_{#3}}{Q^{(#2)}_{#3}(#1)}}
\newcommand{\Kvar}[2][]{
\mathrm{K}(#2)}

\def\densasympletter{D}

\newcommand{\bdensasymp}[4][]{%
\mathbf{\densasympletter}_{\infty,#2}({#3})}

\def\intbdensletter{I}
\newcommand{\intbdens}[3][]{
\mathrm{\intbdensletter}_{#2}(#3)}

\def\regressweights{\mathbf{w}}

\def\useq{a}
\def\ufonc{a^\ast}
\def\vseq{v}
\def\vfonc{v^\ast}

\def\vfoncsym{w^\ast}
\def\limcons{\mathrm{C}}
\def\vphase{\Phi}
\def\dimV{N}
\def\jmean{\mathcal{J}}
\def\allWA{\ref{item:Wreg}--\ref{item:MIM}}

\def\k{\ell}
\def\decim{\gamma}

\title[Asymptotic normality of wavelet estimators] 
{Asymptotic normality of wavelet estimators of the memory parameter for linear processes} 

\author{F. Roueff}
\address{TELECOM ParisTech, INSTITUT Télécom, CNRS LTCI, 46, rue Barrault, 75634 Paris Cédex 13, France.}
\email{roueff@tsi.enst.fr}
\author{M.S. Taqqu}
\address{Department of Mathematics and Statistics, Boston University Boston, MA 02215, USA.}
\email{murad@math.bu.edu}



\date{May 6, 2008}

\thanks{Murad~S.~Taqqu would like to thank l'\'Ecole Normale Sup\'erieure des T\'elecom\-munications in Paris
for their hospitality.  This research was partially supported by the NSF Grants DMS--0505747 and DMS--0706786 at Boston University.}

\begin{document}

\maketitle

\begin{center}
{\it TELECOM ParisTech and Boston University}\\
\end{center}

\begin{abstract}
We consider linear processes, not necessarily Gaussian, with long, short or negative memory. The memory parameter is
estimated semi-parametrically using wavelets  from a sample $X_1,\dots,X_n$ of the process. We treat both the log-regression wavelet estimator and the wavelet Whittle
estimator. We show that these estimators are asymptotically normal as the sample size $n\to\infty$ and we obtain an explicit expression for the limit
variance. 
These results are derived from a general result on the asymptotic normality of the empirical scalogram for linear
processes, conveniently centered and normalized. The scalogram is an array of quadratic forms of the observed sample,
computed from the wavelet coefficients of this sample. In contrast with quadratic forms computed on the Fourier coefficients
such as the periodogram, the scalogram involves correlations which do not vanish as the sample size $n\to\infty$.
\end{abstract}

\renewcommand{\thefootnote}{}
\footnote{\textit{Corresponding author}: F. Roueff, TELECOM ParisTech, 46, rue Barrault, 75634 Paris Cédex 13, France.}


\tableofcontents

\section{Introduction}
We consider  a real-valued process 
  $X \eqdef \{X_k \}_{k\in\Zset}$, not necessarily stationary and for any positive integer $k$, let $\diffop^k X$ denote its
$k$-th order difference. The first order difference is $[\diffop X]_t \eqdef X_t - X_{t-1}$ and $\diffop^k$ is defined recursively.
\begin{definition}[$M(d)$ processes]
The process $X$ is said to have \emph{memory parameter} $d$, $d \in \Rset$ (in short, is an M($d$) process) 
and \emph{short-range} spectral density $f^\ast$ if for any integer $k>d-1/2$,
the $k$-th order difference process $\diffop^{k}X$ is weakly stationary with spectral density function
\begin{equation}\label{eq:spdelta}
f_{\diffop^k X}(\lambda) \eqdef |1-e^{-i\lambda}|^{2(k-d)}\,f^\ast(\lambda)\quad  \lambda\in(-\pi,\pi),
\end{equation}
where $f^\ast$ is a non-negative symmetric function which is continous and non-zero at the origin.
\end{definition}
M($d$) processes encompass both stationary and non-stationary processes, depending on the value of the memory parameter $d$.
The function
\begin{equation}
\label{eq:SpectralDensity:FractionalProcess}
f(\lambda) = |1 - \rme^{-\rmi \lambda}|^{-2d}  f^\ast(\lambda) 
\end{equation}
is called the \emph{generalized spectral density} of $X$. It is a proper spectral density function when $d < 1/2$. In this case,
the process $X$ is covariance stationary with spectral density function $f$. 
The process $X$ is said to have long-memory if $0 < d < 1/2$, short-memory if $d=0$ and negative memory if $d < 0$; the
process is not invertible if $d < -1/2$.
The factor $f^\ast$ is a nuisance function which determine the ``short-range'' dependence. 

In a typical \textit{semiparametric} estimation setting (see for instance
\cite{robinson:1995g, giraitis:robinson:samarov:1997, moulines:soulier:2001}), the following additional assumption is often considered.
\begin{assumption}
\label{H:Lbeta} There exists $\beta\in(0,2]$, $\gamma>0$ and $\varepsilon\in(0,\pi]$ such that for all
$\lambda\in[-\varepsilon,\varepsilon]$,
\begin{equation}\label{eq:Hbeta}
|f^\ast(\lambda)- f^\ast(0) | \leq L \, f^\ast(0) \, |\lambda|^\beta \eqsp .
\end{equation}
Moreover, $f^\ast(0)>0$.
\end{assumption}

We consider an $M(d)$ process satisfying the following \emph{linear assumption}. 
\begin{assumption}\label{assump:linear}
There exists a non-negative integer $K$ such that
\begin{equation}
  \label{eq:defXlinCase}
  [\diffop^KX]_k=\sum_{t\in\Zset} \useq^{(K)}(k-t) \, \xi_t \;,  
\end{equation}
where    $\{\useq^{(K)}(t),\,t\in\Zset\}$ is a real-valued sequence satisfying $\sum_t(\useq^{(K)}(t))^2<\infty$ and
\begin{enum_A}
\item \label{A:1}
 $\{\xi_l,\,l\in\Zset\}$ is a sequence of independent and identically distributed  real-valued
  random variables such that  $\PE[\xi_0]=0$, $\PE[\xi_0^2]=1$ and $\kappa_4\eqdef\PE[\xi_0^4]-3$ is finite. 
\end{enum_A}
\end{assumption}
Here the linear assumption may only apply to a $K$-order increment of $X$ to allow $X$ to be non-stationary.

Our goal is to estimate $d$ by using a Discrete Wavelet Transform (DWT) of $X$.
In order to study the asymptotic properties of the estimator, we use a central limit
theorem for an array of squares of decimated linear processes, established
in~\cite{roueff:taqqu:2008a}, see Theorems~1 and~2 in this reference. Using this result, we extend to the non-Gaussian linear 
processes setting, asymptotic normality 
results for wavelet estimation of the memory parameter $d$, that have been obtained so far for Gaussian processes
(see~\cite[Thoerem~1]{moulines:roueff:taqqu:2006b} and \cite[Theorem~5]{moulines-roueff-taqqu-2007c}). We treat both the
log-regression wavelet estimator and the wavelet Whittle estimator.

In Section~\ref{sec:quadr-forms-wavel}, we give a simplified formulation of the central limit
theorem~\cite[Theorem~2]{roueff:taqqu:2008a} and apply it to the Discrete Wavelet 
Transform setting, obtaining a result on the asymptotic distribution of the scalogram of a linear memory process as the scale
index and the number of observed wavelet coefficients both tend to infinity, see
Theorem~\ref{theo:genericJointConvResult}. 
We then consider two estimators of the memory
parameter $d$, the log--regression wavelet estimator in Section~\ref{sec:log-regr-estim} and the wavelet Whittle estimator in
Section~\ref{sec:wavel-whittle-estim}. Using Theorem~\ref{theo:genericJointConvResult}, we show that both these estimators
are asymptotically normal.

\section{Definition of the empirical scalogram of a finite sample}

We now introduce the wavelet setting and recall the definition of the \emph{scalogram} and the empirical scalogram.
Introduce the functions $\phi(t)$, $t\in\Rset$, and $\psi(t)$, $t\in\Rset$, which will play the role of the father and mother
wavelets respectively, and let $\hat{\phi}(\xi)\eqdef \int_{\Rset} \phi(t)\rme^{-\rmi\xi t}\,\rmd t$ and $\hat{\psi}(\xi)\eqdef \int_{\Rset}
\psi(t)\rme^{-\rmi\xi t}\,\rmd t$ denote their Fourier 
transforms. We suppose that the wavelets $\phi$ and $\psi$ satisfy the following assumptions~:

\begin{enum_W}
\item\label{item:Wreg} $\phi$ and $\psi$ are integrable and have compact supports, $\hat{\phi}(0) = \int_{\Rset} \phi(x) \rmd
  x =1$ and $\int_{\Rset} \psi^2(x) \rmd x = 1$.
\item\label{item:psiHat}
There exists $\alpha>1$ such that
$\sup_{\xi\in\Rset}|\hat{\psi}(\xi)|\,(1+|\xi|)^{\alpha} <\infty$,
\item\label{item:MVM} The function $\psi$ has  $M$ vanishing moments, \ie\ $ \int_{\Rset} t^l \psi(t) \,\rmd t=0$ for all $l=0,\dots,M-1$
\item\label{item:MIM} The function $ \sum_{k\in\Zset} k^l\phi(\cdot-k)$
is a polynomial of degree $l$ for all $l=0,\dots,M-1$.
\end{enum_W}

We now define what
we call the DWT of $X$ in discrete time. Define the family $\{\psi_{j,k}, j > 0, k \in \Zset\}$ of translated and
dilated functions 
\begin{equation}\label{eq:psiJK}
\psi_{j,k}(t)=2^{-j/2}\,\psi(2^{-j}t-k) \eqsp.
\end{equation}
Using the scaling function $\phi$, we first define the functions
\begin{equation}\label{eq:bX}
\bX_n(t) \eqdef \sum_{k=1}^n X_k \,\phi(t-k) \quad\text{and}\quad \bX(t) \eqdef \sum_{k\in\Zset} X_k\, \phi(t-k)
\end{equation}
The (details) wavelet coefficients are then defined as follows, for all $j \geq 0, k \in \Zset$,
\begin{align}\label{eq:coeff}
\dwt_{j,k} \eqdef \int_{-\infty}^\infty \bX(t) \psi_{j,k}(t)\,\rmd t.
\end{align}
These wavelet coefficients are the DWT of $X$. If the support of the scaling function $\phi$ is included in $(-\L,0)$ for
some integer $\L\geq1$, then $\bx_n(t)=\bx(t)$ for all $t= 0,\dots, n-\L+1$. 
If the support of the wavelet function $\psi$ is included in $(0,\L)$, then, the support of $\psi_{j,k}$ is included in the
interval $(2^j k, 2^j(k+\L) )$. Hence
\begin{equation}\label{eq:coeffN}
\dwt_{j,k}=\int_{-\infty}^\infty \bX_n(t) \psi_{j,k}(t)\,\rmd t \; ,
\end{equation}
for all $(j,k)\in\indexset_n$, where
\begin{equation}
\label{eq:deltan}
\indexset_n \eqdef \{(j,k):\,j\geq0, 0\leq k  \leq 2^{-j}(n-\L+1)-\L \} \eqsp.
\end{equation}

For any $j$, the wavelet coefficients $\{\dwt_{j,k} \}_{k \in \Zset}$ are obtained by discrete convolution and downsampling.
More precisely, under \ref{item:Wreg}, for all $j \geq 0$, $k \in \Zset$,
\begin{equation}\label{eq:down}
\dwt_{j,k}=\sum_{l\in\Zset} x_l\,h_{j,2^jk-l}=(h_{j,\cdot}\star X)_{2^jk}= (\downarrow^j [h_{j,\cdot}\star X])_{k},
\end{equation}
where $ h_{j,l} \eqdef 2^{-j/2} \int_{-\infty}^\infty \phi(t+l)\psi(2^{-j}t)\,\rmd t$,  $\star$ denotes the convolution  of discrete sequences and,
for any sequence $\{ c_k \}_{k \in \Zset}$, $(\downarrow^j c)_k = c_{k 2^j}$.
For all $j\geq0$,  $H_j(\lambda)\eqdef\sum_{l\in\Zset}h_{j,l}\rme^{-\rmi\lambda l}$ denotes the discrete Fourier transform of $\{ h_{j,l} \}_{l \in \Zset}$,
\begin{equation}\label{eq:fjdef}
H_j(\lambda) \eqdef 2^{-j/2} \int_{-\infty}^\infty \sum_{l\in\Zset}\phi(t+l)\rme^{-\rmi\lambda l} \psi(2^{-j}t)\,\rmd t.
\end{equation}
For all $j\geq0$ and all $m=0,\dots,M-1$,
\[
\sum_{l\in\Zset} h_{j,l}\,l^m= 2^{-j/2} \int_{-\infty}^\infty \psi(2^{-j}t) \sum_{l\in \Zset} \phi(t+l) l^m \rmd t \eqsp.
\]
Under assumption \ref{item:MIM}, $t \mapsto \sum_{l \in \Zset} \phi(t+l) l^m$ is a polynomial of degree $m$ and \ref{item:MVM} therefore
implies that $\sum_{l \in \Zset} h_{j,l} \, l^m= 0$; equivalently, the trigonometric polynomial $H_j$ satisfies
$\left. \frac{\rmd^m H_j(\lambda)}{\rmd\lambda^m} \right|_{\lambda=0} = 0$, $m=0, \dots, M-1$ and thus admits a zero at $0$ of degree at least equal to $M$.
Therefore,  $H_j(\lambda)$ can be factorized as 
\begin{equation}
  \label{eq:HHtilde}
  H_j(\lambda)= (1-\rme^{\rmi \lambda})^M \tilde{H}_{j}(\lambda) \; ,
\end{equation}
where $\tilde{H}_{j}(\lambda)$ is a trigonometric polynomial.
Hence, the wavelet coefficient \eqref{eq:down} may be computed as
\begin{equation}
\label{eq:op}
\dwt_{j,k}= (\downarrow^j [\tilde{h}_{j,\cdot}\star \diffop^M X])_{k}
\end{equation}
where $ \{ \tilde{h}_{j,l} \}_{l \in \Zset}$ are the coefficients of the trigonometric polynomial $\tilde{H}_j$.

Let then $\{\phi,\psi\}$ be a pair of scale function and wavelet satisfying \allWA.  
Let $X=\{X_k,\;k\in\Zset\}$ be a process 
such that  $\diffop^M X$ is weakly stationary and define the DWT  $\{\dwt_{j,k},\;j\geq0,\;k\in\Zset\}$ of $X$ by~(\ref{eq:coeff}).
By~(\ref{eq:op}), $\{\dwt_{j,k},\;k\in\Zset\}$ is a weakly stationary process for all scales $j\geq0$.
\begin{definition}\label{def:scalogram}
The \emph{scalogram} of $X$ is the non-negative sequence $\{\vj{j}{}2,\;j\geq0\}$ of variances of
$\{\dwt_{j,k},\;k\in\Zset\}$, namely
\begin{equation}
\label{eq:definitionvj}
\vj{j}{d,f^{\ast}}2 \eqdef \PVar [ \dwt_{j,0}]=\PE \left[ \dwt_{j,0}^2 \right],\quad j\geq0\eqsp.
\end{equation}
\end{definition}
\begin{remark}
Observe that, under Assumption~\ref{assump:linear}, if $M\geq K$, then $\diffop^M X$ is a centered weak stationary process
and the scalogram of $X$ is well defined. 
\end{remark}

Wavelet estimators of the memory parameter $d$ are typically
based on quadratic forms of the wavelet coefficients. 
This is reasonable because, for large scale $j$, $\log\vj{j}{d,f^\ast}2$ is  approximately an affine function of 
$j$ with slope $(2 \log2)\,d$ (see~\cite{moulines-roueff-taqqu-2007a}) and,
given $n$ observations $X_1, \dots, X_n$, $\vj{j}{d,f^\ast}2$
can be estimated by the empirical second moment
\begin{equation}\label{eq:hvj}
\hvj{j}{n} \eqdef n_j^{-1} \sum_{k=0}^{n_j-1} \dwt_{j,k}^2 \eqsp ,
\end{equation}
which is a quadratic form on the wavelet coefficients. Here we denote by $n_j$ the number of available wavelet
coefficients at scale index $j$, namely, from~(\ref{eq:deltan}), 
\begin{equation}\label{eq:ni}
n_j = [2^{-j} (n - \L + 1) - \L +1] \; ,
\end{equation}
where $\L$ is the size of the time series and $[x]$ denotes the integer part of $x$. It is important to note that
although the wavelet coefficient $\dwt_{j,k}$ does not depend on $n$, the empirical second moment $\hvj{j}{n}$ does through $n_j$.

\begin{definition}\label{def:empscal}
Let $\{\phi,\psi\}$ be a pair of scale function and wavelet satisfying \allWA\ and $n\geq1$. Let us denote the maximal scale
index  $\maxscale=\maxscale(n)$ by   
\begin{equation}
  \label{eq:defMaxscale}
\maxscale\eqdef\max\{j:\;n_j>0\}=\left\lceil\log_2\left(\frac{n-\L+1}{\L}\right)\right\rceil\;,
\end{equation}
where $n_j$ is defined by~(\ref{eq:ni}).
The \emph{empirical scalogram} of the sample  $\{X_1,\dots,X_n\}$ is the non-negative process $\{\hvj{j}{n},\;j\geq0\}$, where
\begin{enumerate}[$\bullet$]
\item  for all $j=1,\dots,\maxscale$, $\hvj{j}{n}$ is defined by~(\ref{eq:hvj}),
\item and by convention, $\hvj{j}{n}=0$ for $j>\maxscale$.
\end{enumerate}
\end{definition}

The estimator of the memory parameter $d$ can then be obtained as follows~:
\begin{enumerate}[(1)]
\item by
regressing the logarithm of the empirical variance $\log (\hvj{i}{n})$  for a finite number of scale indices
$j \in \{\lowscale, \dots, \upscale\}$
where $\lowscale$ is the lower scale and $\upscale\leq\maxscale$ is the upper scale in the regression, see Section~\ref{sec:log-regr-estim}. 
\item by minimizing a contrast derived from the likelihood of an array of independent Gaussian random variables 
each row $j\in \{\lowscale, \dots, \upscale\}$ of which having  empirical variance $\log (\hvj{i}{n})$, see Section~\ref{sec:wavel-whittle-estim}. 
\end{enumerate}

\section{Joint weak convergence of the empirical scalogram of a linear process}
\label{sec:quadr-forms-wavel}

We let $\cl$ denote the convergence in law. 
For convenience, we first state a CLT based on results  of~\cite{roueff:taqqu:2008a}.

\begin{theorem}\label{thm:CLTlocal}
Let $\{\vseq_{i,j}(t),\;t\in\Zset\}$ be  real-valued sequences satisfying $\sum_{t\in\Zset} \vseq_{i,j}^2(t)<\infty$ for
all $i=1,\dots,\dimV$ and $j\geq0$. 
Suppose that there exist $\delta>1/2$, $\varepsilon\in(0,\pi]$, a sequence of $[-\pi,\pi)$-valued functions $\vphase_j(\lambda)$
defined on $\lambda\in\Rset$ and continuous functions $\vfonc_{i,\infty}$, $i=1,\dots,\dimV$, defined on $\Rset$ such that     
\begin{align}
\label{eq:unfiBoundvfoncLocale}
&\sup_{j\geq0} \sup_{|\lambda|\leq\varepsilon}
2^{-j/2}|\vfonc_{i,j}(\lambda)|(1+2^j\left|\lambda\right|)^\delta < \infty \;,  \\
\label{eq:Limitvfonc}
&\lim_{j\to\infty}2^{-j/2}\vfonc_{i,j}(2^{-j}\lambda)\rme^{\rmi\vphase_j(\lambda)} = \vfonc_{i,\infty}(\lambda) \quad\text{for all}\quad
\lambda\in\Rset\;,\\
\label{eq:BoundvfoncLocaleNonZero}
&n_j^{1/2}\;\int_{0}^\pi\1(|\lambda|>\varepsilon)\;
|\vfonc_{i,j}(\lambda)|^2\;\rmd\lambda \to0\quad\text{as}\quad j \to\infty \;,
\end{align}
where $\vfonc_{i,j}$ denotes the Fourier series associated to the sequence  $\{\vseq_{i,j}(t),\;t\in\Zset\}$,
\begin{equation}
\label{eq:vfoncdef}
\vfonc_{i,j}(\lambda)=(2\pi)^{-1/2}\,\sum_{t\in\Zset} \vseq_{i,j}(t)\,\rme^{-\rmi\lambda t} \; .
\end{equation}
Define $\{Z_{i,j,k},\;i=1,\dots,\dimV, j\geq0, k\in\Zset\}$ as
\begin{equation} 
\label{eq:Zdef}
Z_{i,j,k} = \sum_{t\in\Zset} \vseq_{i,j}(2^jk-t) \, \xi_t  \;,\quad i=1,\dots,\dimV,\,k\in\Zset,\,j\geq0\;,
\end{equation}
where $\{\xi_t,\;t\in\Zset\}$ satisfies~\ref{A:1}. Then, for any diverging sequence $(n_j)$,  as $j\to\infty$,
\begin{equation}
\label{eq:CenteredZ}
n_j^{-1/2} \sum_{k=0}^{n_j-1} 
\left[
\begin{array}{c}
Z_{1,j,k}^2 -\PE[Z_{1,j,k}^2]\\
\vdots\\
Z_{\dimV,j,k}^2 -\PE[Z_{\dimV,j,k}^2]
\end{array}
\right]\cl\calN(0,\Gamma)\;,
\end{equation}
where $\Gamma$ is the covariance matrix defined by
\begin{equation}
  \label{eq:GammaDef}
  \Gamma_{i,i'}= 
4\pi \,\int_{-\pi}^\pi \left|\sum_{p\in\Zset}\vfonc_{i,\infty}\overline{\vfonc_{i',\infty}}(\lambda+2p\pi)\right|^2 \, \rmd\lambda\;,\quad 1\leq i, i'\leq \dimV\; .
\end{equation}
Moreover, one has
\begin{equation}
    \label{eq:varianceZlim}
\lim_{j\to\infty}\PE\left[Z_{i,j,0}^2\right]=\int_{-\infty}^{\infty}\left|\vfonc_{i,\infty}(\lambda)\right|^2\,\rmd\lambda <
\infty \; .
\end{equation}
\end{theorem}

\begin{proof}
Observe first that we allowed $\varepsilon=\pi$, in which case Condition~(\ref{eq:BoundvfoncLocaleNonZero}) is always satisfied since
the integral vanishes for all $j\geq0$. 
The CLT~(\ref{eq:CenteredZ}) is a strict application of  Theorem~1 in~\cite{roueff:taqqu:2008a} for
$\varepsilon=\pi$ and Theorem~2 in~\cite{roueff:taqqu:2008a} for  $\varepsilon<\pi$. 
Using the notations of~\cite{roueff:taqqu:2008a} we have here
$\decim_j=2^j$ for all $j\geq0$ and $\lambda_{i,\infty}=0$ for all $i=1,\dots,\dimV$.
Observe that in this case, in these two theorems, $\limcons_{i,i'}=1$ and
$\vfoncsym_{i,i'}=\vfonc_{i,\infty}\overline{\vfonc_{i',\infty}}$ for all $i,i'=1,\dots,\dimV$. 
The limit~(\ref{eq:varianceZlim})
follows from   Relation~(15) in Proposition~1 in the same paper, see also Remark~11.
\end{proof}  
\begin{remark}
When the convergence rate to the limit~(\ref{eq:varianceZlim}) is fast enough, we will be able to replace
the expectations in~(\ref{eq:CenteredZ}) by 
$\int_{-\infty}^{\infty}\left|\vfonc_{i,\infty}(\lambda)\right|^2\,\rmd\lambda$, $i=1,\dots,\dimV$, which does not depend on $j$.  
\end{remark}

In the Gaussian case the  spectral density of the wavelet coefficients is
sufficient for studying the quadratic forms~(\ref{eq:hvj}) since second order properties fully determine the distribution of the wavelet
coefficients. 
This is not so in the linear case, which we consider in this paper. 
It is possible, however, to establish a multivariate
central limit theorem for the empirical scalogram.
 
\begin{theorem}\label{theo:genericJointConvResult}
Let $X$ be an $M(d)$ process with short-range spectral density $f^\ast$ 
and suppose that Assumption~\ref{assump:linear} holds.  
Assume that~\allWA\ hold with
\begin{equation}
  \label{eq:CondMalphadKsansBeta}
   1/2-\alpha< d\leq M\quad\text{and}\quad K\leq M \;.
\end{equation}
Let $\lowscale=\lowscale(n)$ be a scale index depending on $n$ such that $\lowscale(n)\to\infty$ and
$n2^{-\lowscale(n)}\to\infty$ as $n\to\infty$. 
Assume that one of the two following conditions hold.
\begin{align}
\label{eq:CondFstarBounded}
&\sup_{\lambda\in(-\pi,\pi)}f^\ast(\lambda)<\infty\\
\label{eq:CondFstarUNBounded}
& (n2^{-\lowscale(n)})^{1/2}2^{\lowscale(n)(1-2\alpha-2d)}\to0\quad\text{as $n\to\infty$} \;.
\end{align}
Then, as $n\to\infty$, one has  the folowing central limit:
\begin{equation}
\label{eq:ScalogramLimit}
\left\{\sqrt{n2^{-\lowscale(n)}}2^{-2\lowscale(n) d}(\hvj{\lowscale(n)+\dj}{n_{\lowscale(n)+\dj}}-\vj{\lowscale(n)+\dj}{d,f^\ast}{2}),\;\dj\geq0\right\}
\cl \left\{\hvjLIM{d}{\dj},\;\dj\geq0\right\} \; ,
\end{equation}
where $\hvjLIM{d}{\centerdot}$ denotes a centered Gaussian process  
defined on $\Nset$ with covariance function 
$\AVvarJoint{d}{\dj,\dj'}=\PCov\left(\hvjLIM{d}{\dj},\hvjLIM{d}{\dj'}\right)$, $\dj,\dj'\geq0$, given by
\begin{equation}
\label{eq:CovMatrix}
\AVvarJoint{d}{\dj,\dj'}\eqdef
4 \pi \, (f^\ast(0))^2 \, 2^{4d( \dj\vee\dj') + \dj \wedge \dj'} \,
\int_{-\pi}^{\pi} \left|\bdensasymp{|\dj-\dj'|}{\lambda}{d} \right|^2 \, \rmd\lambda\eqsp ,
\end{equation}
with, for all $\dj\geq0$ and $\lambda\in(-\pi,\pi)$,
\begin{equation}\label{eq:bDpsi}
\bdensasymp[\psi]{\dj}{\lambda}{d} \eqdef
\sum_{l\in\Zset} |\lambda+2l\pi|^{-2d}\,\be_{\dj}(\lambda+2l\pi) \,
\overline{\hat{\psi}(\lambda+2l\pi)}\hat{\psi}(2^{-\dj}(\lambda+2l\pi)).
\end{equation}
and $\be_\dj(\xi) \eqdef 2^{-\dj/2}[1, \rme^{-\rmi2^{-\dj}\xi}, \dots, \rme^{-\rmi(2^{\dj}-1)2^{-\dj}\xi}]^T$.
\end{theorem} 
\begin{remark}
  We assume $1-2\alpha-2d<0$
  in~(\ref{eq:CondMalphadKsansBeta}) so that~(\ref{eq:CondFstarUNBounded}) imposes a sufficiently fast growth rate on
  $\lowscale(n)$ as $n\to\infty$. On the other hand this rate has to be slow enough for the assumption
  $n2^{-\lowscale(n)}\to\infty$ to hold. 
\end{remark}
\begin{proof}
In~(\ref{eq:defXlinCase}), the sequence $\{\useq^{(K)}(t),\;t\in\Zset\}$ depends on $K$. To define a quantity which does
not, we go to the Fourier domain and set
$$
\ufonc(\lambda)\eqdef(2\pi)^{-1/2}\,(1-\rme^{-\rmi\lambda})^{-K}\,\sum_{t\in\Zset} \useq^{(K)}(t) \, \rme^{-\rmi\lambda t} \, ,
$$
where the sum over $t\in\Zset$ converges in the sense of $L^2(-\pi,\pi)$. This function $\ufonc(\lambda)$ satisfies 
\begin{equation}
  \label{eq:ufoncANDf}
|\ufonc(\lambda)|^2 = |1-e^{-\rmi\lambda}|^{-2d} \, f^\ast(\lambda)= f(\lambda) \; ,   
\end{equation}
where $f$ is defined in~(\ref{eq:SpectralDensity:FractionalProcess}).
Moreover, by~(\ref{eq:op}), since $K\leq M$ (see Condition~(\ref{eq:CondMalphadKsansBeta})), for all $j\in\Nset$ and
$k\in\Zset$, the wavelet coefficients of $X$ can be expressed as 
$$
\dwt_{j,k} = ( \downarrow^j [\tilde{h}_{j,\cdot}\star \diffop^{M-K} ( \useq^{(K)} \star \xi ) ])_k \; .
$$
Since $\tilde{h}_{j,\cdot}$ is a finite sequence, we obtain that
\begin{equation}
  \label{eq:wavLinDecim}
\dwt_{j,k} = \sum_{t\in\Zset} \useq_{j}(k2^j-t) \, \xi_t \;,
\end{equation}
where $\{\useq_{j}(t),\,t\in\Zset\}$ is the sequence $\tilde{h}_{j,\cdot}\star \diffop^{M-K} ( \useq^{(K)})$ which is characterized by the
$L^2(-\pi,\pi)$ converging series 
\begin{equation}
  \label{eq:ufoncJDef}
\ufonc_j(\lambda)\eqdef
(2\pi)^{-1/2}\,\sum_{t\in\Zset} \useq_{j}(t)\,\rme^{-\rmi\lambda t} 
= \tilde{H}_j(\lambda)(1-\rme^{-\rmi\lambda})^{M}\ufonc(\lambda) \; ,
\end{equation}
which, in view of~(\ref{eq:HHtilde}), can be simply written as
\begin{equation}
  \label{eq:ufoncJ}
\ufonc_j(\lambda)=H_j(\lambda)\ufonc(\lambda)\;.
\end{equation}

To prove the theorem, we need to show that, for any integer $\ell\geq0$, one has
\begin{equation}
\label{eq:JointCentralLimitEmpVar}
\sqrt{n2^{-\lowscale}}2^{-2\lowscale d} \left(\left[
\begin{array}{c}
\hvj{\lowscale}{n_{\lowscale}}-\vj{\lowscale}{d,f^\ast}{2}\\
\hvj{\lowscale+1}{n_{\lowscale+1}}-\vj{\lowscale+1}{d,f^\ast}{2}\\
\vdots\\
\hvj{\lowscale+\ell}{n_{\lowscale+\ell}}-\vj{\lowscale+\ell}{d,f^\ast}{2}
\end{array}
\right] \right) \cl \calN\left(0, \left[\AVvarJoint{d}{\dj,\dj'}, \,\dj,\dj'=0,\dots,\ell\right] \right) \; .
\end{equation}
To this end, we will apply Theorem~\ref{thm:CLTlocal} by relating the right-hand side of~(\ref{eq:JointCentralLimitEmpVar}) to the
right-hand side of~(\ref{eq:CenteredZ}) and by expressing the empirical scalogram $\hvj{\lowscale+\dj}{n_{\lowscale}+\dj}$, $0\leq\dj\leq\ell$
in terms of $Z_{i,j,k}$ with adapted indices $j,k$ and $i$.

We let  $j=\lowscale+\ell$, that is $j$ is the maximal scale in~(\ref{eq:JointCentralLimitEmpVar}). We let $k$ take values
$k=0,\dots,n_j$, where $n_j$, given by~(\ref{eq:ni}) is the number of wavelet coefficients available at the maximal scale
$j$. In Theorem~\ref{thm:CLTlocal}, $Z_{i,j,k}$ is viewed as the $i^{th}$ component of a $k$-wise stationary vector, with
$i=1,\dots,\dimV$. In order to recover this stationarity from the set of wavelet coefficients used to compute the empirical
variances in~(\ref{eq:JointCentralLimitEmpVar}), we do as follows.
We represent $i$ as  $i=2^{\ell-\dj}+\dk$ where $\dj\in\{0,\dots,\ell\}$ and
$\dk\in\{0,\dots,2^{\ell-\dj}-1\}$ and let $\dimV=\sum_{\dj=0}^\ell 2^{\ell-\dj}=2^{\ell+1}-1$. For each $(\dj,\dk)$, we set
$i=2^{\ell-\dj}+\dk$, $j=\lowscale+\ell$, and
\begin{equation}
  \label{eq:vseqDefWav}
\vseq_{i,j}(t) \eqdef 2^{-\lowscale d}\useq_{\lowscale+\dj}(t+\dk2^{\lowscale+\dj}),\quad t\in\Zset \;,
\end{equation}
where $\useq_{\lowscale+\dj}$ is defined in~(\ref{eq:wavLinDecim}). Thus if we focus on a scale $j'\geq \lowscale$ and express it
as $j'=\lowscale+\dj=j-\ell+\dj$ (see Figure~\ref{fig:jLjdjj'}), we have
\begin{equation}
  \label{eq:vw}
\vseq_{i,j}(t) = 2^{-\lowscale d}\useq_{j'}(t+\dk2^{j'}),\quad t\in\Zset \;.
\end{equation}
Hence, by definition of $Z_{i,j,k}$ in~(\ref{eq:Zdef}), one has
\begin{align}
\nonumber
   Z_{i,j,k} &=\sum_{t\in\Zset} \vseq_{i,j}(2^j k- t) \xi_t \\
\nonumber
&=2^{-\lowscale d}\sum_{t\in\Zset} \useq_{j'}(2^{j'}\{2^{\ell-\dj} k+\dk\}-t) \, \xi_t\\
\label{eq:ZandDWT}
&=2^{-\lowscale d}\dwt_{j',2^{\ell-\dj} k+\dk}\;.
\end{align}
By~(\ref{eq:ni}),
$$
n_{j'}=2^{-(j-\ell+\dj)}(n-\L+1)-\L+1= 2^{\ell-\dj}\;n_j+(\L-1)(2^{\ell-\dj}-1),
$$
and hence by~(\ref{eq:definitionvj}),~(\ref{eq:hvj}) and~(\ref{eq:ZandDWT}),
\begin{align}
\nonumber
2^{-2\lowscale d}\left(\hvj{j'}{n_{j'}}-\vj{j'}{d,f^\ast}{2}\right)&=
n_{j'}^{-1} 2^{-2\lowscale d}\sum_{k'=0}^{n_{j'}-1}(\dwt_{j',k'}^2-\PE[\dwt_{j',k'}^2])\\
\label{eq:scaloLinearProc}
&=n_{j'}^{-1} \sum_{\dk=0}^{2^{\ell-\dj}-1} \left(\sum_{k=0}^{n_j-1}\left\{Z_{i,j,k}^2-\PE[Z_{i,j,k}^2]\right\}\right)+R_{j'} \;,
\end{align}
where $j'=j-\ell+\dj$, $k'=2^{\ell-\dj} k+\dk$, $i=2^{\ell-\dj}+\dk$ and
\begin{equation}
  \label{eq:RjprimDef}
  R_{j'}\eqdef n_{j'}^{-1}\sum_{\dk=0}^{(\L-1)(2^{\ell-\dj}-1)-1}
  \left\{Z_{i,j,n_j}^2-\PE\left[Z_{i,j,n_j}^2\right] \right\}\;.
\end{equation}
We then have 
\begin{equation}\label{eq:MatrixExprDWTwrtZ}
2^{-2\lowscale d}\left[
\begin{array}{c}
\hvj{\lowscale}{n_{\lowscale}}-\vj{\lowscale}{d,f^\ast}{2}\\
\hvj{\lowscale+1}{n_{\lowscale+1}}-\vj{\lowscale+1}{d,f^\ast}{2}\\
\vdots\\
\hvj{\lowscale+\ell}{n_{\lowscale+\ell}}-\vj{\lowscale+\ell}{d,f^\ast}{2}
\end{array}
\right]
= A_n \left[
\begin{array}{c}
\sum_{k=0}^{n_j-1} \{Z_{1,j,k}^2 -\PE[Z_{1,j,k}^2]\} \\
\sum_{k=0}^{n_j-1} \{Z_{2,j,k}^2 -\PE[Z_{1,j,k}^2]\} \\
\vdots\\
\sum_{k=0}^{n_j-1} \{Z_{\dimV,j,k}^2 -\PE[Z_{\dimV,j,k}^2]\}
\end{array}
\right] 
+
\left[
\begin{array}{c}
R_{\lowscale}\\
R_{\lowscale+1}\\
\vdots\\
R_{\lowscale+\ell}
\end{array}
\right]   \;,
\end{equation}
where
\begin{equation}
  \label{eq:AnDef}
A_n\eqdef  \left[
\begin{array}{ccc}
0\dots&\dots0&\overbrace{n_{\lowscale}^{-1}\dots n_{\lowscale}^{-1}}^{2^\ell\text{ times }}\\
0\dots0&\overbrace{n_{\lowscale+1}^{-1}\dots n_{\lowscale+1}^{-1}}^{2^{\ell-1}\text{ times }}&0\dots0\\
\vdots&\vdots&\vdots\\
n_{\lowscale+\ell}^{-1}\;\;0\dots0&0\dots0&0\dots0
\end{array}
\right]
\end{equation}
is an $(\ell+1)\times \dimV$ matrix. The entries are $n_{j'}^{-1}=n_{\lowscale+\dj}^{-1}$, where $\dj$ goes from 0 (top line) to
$\ell$ (bottom line), see Figure~\ref{fig:jLjdjj'}.
\begin{figure}[h]
  \centering \setlength{\unitlength}{1mm}
\begin{picture}(100,40)(-50,0)
\put(-40,12){\line(1,0){80}}
\put(-40,10){\line(0,1){3}}
\put(-40,15){$\lowscale$}
\put(-10,10){\line(0,1){3}}
\put(-10,15){$j'$}
\put(40,10){\line(0,1){3}}
\put(40,15){$\lowscale+\ell=j$}
\put(-40,22){\vector(1,0){30}}
\put(-10,22){\vector(-1,0){30}}
\put(-35,25){$j'-\lowscale=\dj$}
\put(-10,22){\vector(1,0){50}}
\put(40,22){\vector(-1,0){50}}
\put(0,25){$j-j'=\ell-\dj$}
\end{picture}
\caption{This figure indicates the relationship between the various variables.}\label{fig:jLjdjj'}
\end{figure}
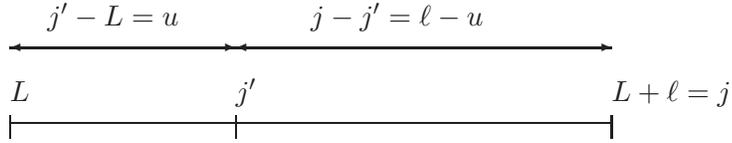

Let us check that the assumptions of Theorem~\ref{thm:CLTlocal} apply to~(\ref{eq:vseqDefWav}), that is, we
show that~\eqref{eq:unfiBoundvfoncLocale},\eqref{eq:BoundvfoncLocaleNonZero} and~(\ref{eq:Limitvfonc}) are verified for
$\vseq_{i,j}(t)$ defined by~(\ref{eq:vseqDefWav}).
Using~(\ref{eq:vw}),~(\ref{eq:vfoncdef}),~(\ref{eq:ufoncJDef}) and~(\ref{eq:ufoncJ}), we get
\begin{align*}
\vfonc_{i,j}(\lambda)&=2^{-\lowscale d}\;\rme^{\rmi \lambda\dk2^{j'}}\;\ufonc_{j'}(\lambda)\\
&= 2^{-\lowscale d}\;\rme^{\rmi \lambda\dk2^{j'}}\;H_{j'}(\lambda)\ufonc(\lambda) \;.  
\end{align*}
By  continuity of $f^\ast$ at the origin we have $\sup_{\lambda\in(-\varepsilon,\varepsilon)}\sqrt{f^\ast(\lambda)}\leq C$
for some $C>0$ and $\varepsilon\in(0,\pi]$. Morever, under  \eqref{eq:CondFstarBounded}, we may set $\varepsilon=\pi$.
By~(\ref{eq:ufoncANDf}) and~(\ref{eq:SpectralDensity:FractionalProcess}) we get, for all
$\lambda\in(-\varepsilon,\varepsilon)$, 
$$
\left|\ufonc(\lambda)\right|\leq |1-\rme^{-\rmi\lambda}|^{-d}\;\sqrt{C}
\leq \sqrt{C}\; |\lambda|^{-d}\;.
$$
By \cite[Proposition~3]{moulines-roueff-taqqu-2007a}, we have, for all $\lambda\in(-\pi,\pi)$,
$$
\left|H_{j'}(\lambda)-2^{j'/2}\hat{\phi}(\lambda)\overline{\hat{\psi}(2^{j'}\lambda)}\right|\leq C\;2^{j'(1/2-\alpha)}|\lambda|^M
$$
and
$$
\left|H_{j'}(\lambda)\right|\leq C\;2^{j'/2}\;|2^{j'}\lambda|^M\;(1+2^{j'}|\lambda|)^{-\alpha-M} \;,
$$
where $C$ is a positive constant and $\alpha$ and $M$ are defined in~\ref{item:psiHat} and~\ref{item:MVM}, respectively.
Using that $j=\lowscale+\ell$ and $j'=\lowscale+\dj$ for some $\dj$ only depending on $i$ and $\lowscale,j'=j+O(1)$, the last 4 displays
and Condition~(\ref{eq:CondMalphadKsansBeta}) easily yield
\begin{align}\label{eq:localBoundvfoncCLTwav}
&|\vfonc_{i,j}(\lambda)|\leq C\;2^{j/2}\;(1+2^j|\lambda|)^{-\alpha-d}\quad \text{for all $\lambda\in(-\varepsilon,\varepsilon)$} \\
\label{eq:localBoundvfoncCLTwav2}
&\int_{\varepsilon}^\pi |\vfonc_{i,j}(\lambda)|^2 \;\rmd\lambda \leq C\; 2^{\lowscale(1-2\alpha-2d)}\int_{\varepsilon}^\pi f^\ast(\lambda) \;\rmd\lambda \;,
\end{align}
and the bound
\begin{multline}\label{eq:localBoundvfoncCLTwav3}
\left|2^{-j/2}\vfonc_{i,j}(2^{-j}\lambda)-
2^{(\dj-\ell)/2-\lowscale d}\;\rme^{\rmi
  \lambda\dk2^{\dj-\ell}}\;\hat{\phi}(2^{-j}\lambda)\overline{\hat{\psi}(2^{\dj-\ell}\lambda)}\ufonc(2^{-j}\lambda)\right|\\
=2^{-j/2-\lowscale d}\;\left|\ufonc(2^{-j}\lambda)\right|\;
\left|H_{j'}(2^{-j}\lambda)-2^{j'/2}\hat{\phi}(2^{-j}\lambda)\overline{\hat{\psi}(2^{\dj-\ell}\lambda)}\right|\\
\leq 
C\;2^{-j(\alpha+M)}\;|\lambda|^{M-d} \;,  
\end{multline}
valid for $2^{-j}|\lambda|\leq\varepsilon$ with $C$ denoting some positive constant depending neither on $\lambda$ nor
on $j\geq0$. Relation~(\ref{eq:localBoundvfoncCLTwav}) is~\eqref{eq:unfiBoundvfoncLocale} with $\delta=\alpha+d>1/2$.
Under  \eqref{eq:CondFstarBounded}, $\varepsilon=\pi$ and \eqref{eq:BoundvfoncLocaleNonZero} trivially holds (see the proof
of Theorem~\ref{thm:CLTlocal}).
Otherwise, since~(\ref{eq:ni}) and $n2^{-\lowscale}\to\infty$ imply $n_j\sim n2^{-j}=n2^{-\lowscale-\ell}$,
Relations~\eqref{eq:localBoundvfoncCLTwav2}, \eqref{eq:CondFstarUNBounded} and the fact that $f^\ast$ is always integrable away of the origin 
(since $|1-\rme^{-\rmi\lambda}|^{K-d}f^\ast(\lambda)$ is a spectral density and $|1-\rme^{-\rmi\lambda}|^{K-d}$ is lower bounded for $\lambda$ away of zero) 
imply \eqref{eq:BoundvfoncLocaleNonZero}.
By~\ref{item:Wreg}, $\hat{\phi}$ is continuous at the origin where it takes value 1 and
using \eqref{eq:localBoundvfoncCLTwav3},~(\ref{eq:ufoncANDf}),~(\ref{eq:SpectralDensity:FractionalProcess}) and the continuity of 
$f^\ast(\lambda)$ at $\lambda=0$, we have, for all $\lambda\in\Rset$,
$$
\hat{\phi}(2^{-j}\lambda)2^{-\lowscale d}|\ufonc(2^{-j}\lambda)|\to 2^{\ell d}\; \sqrt{f^\ast(0)}\; |\lambda|^{-d}
\quad\text{as}\quad j\to\infty \;.
$$
Hence we obtain~(\ref{eq:Limitvfonc}) with
\begin{equation}
  \label{eq:vfoncInftyDWTcase}
\vfonc_{i,\infty}(\lambda)=2^{(\dj-\ell)/2+\ell d}\;\rme^{\rmi
  \lambda\dk2^{\dj-\ell}}\; \sqrt{f^\ast(0)}\; |\lambda|^{-d} \;\overline{\hat{\psi}(2^{\dj-\ell}\lambda)}
\end{equation}
for all  $i=2^{\ell-\dj}+\dk$ such that $\dj\in\{0,\dots,\ell\}$ and $\dk\in\{0,\dots,2^{\ell-\dj}-1\}$, and
$$
\rme^{\rmi\vphase_j(\lambda)}\eqdef\frac{\overline{\ufonc(\lambda)}}{|\ufonc(\lambda)|}\;.  
$$
Since~\eqref{eq:unfiBoundvfoncLocale},\eqref{eq:BoundvfoncLocaleNonZero} and~(\ref{eq:Limitvfonc}) hold and $n_j\to\infty$, 
we may apply Theorem~\ref{thm:CLTlocal} and obtain~(\ref{eq:CenteredZ}). 

Observe that~(\ref{eq:ni}) and $n2^{-\lowscale}\to\infty$ imply, for $j'=\lowscale+\dj\in\{\lowscale,\dots,\lowscale+\ell\}$,  
$$
n_{j'}=n2^{-j'}+O(1)\sim n2^{-\lowscale-\dj} \to\infty\quad\text{as}\quad n\to\infty\; .
$$
Since $(n2^{-\lowscale})^{1/2}n_{j'}^{-1}\sim(n2^{-\lowscale})^{1/2}(n2^{-\lowscale-\dj})^{-1}=
(n2^{-j})^{-1/2}2^{-\ell/2}2^{\dj}\sim n_j^{-1/2}2^{-\ell/2}2^{\dj}$,~(\ref{eq:AnDef}) yields, as $n\to\infty$,
$$
\sqrt{n2^{-\lowscale}}\;A_n \sim  n_j^{-1/2} \;  2^{-\ell/2}\; A_\infty \quad\text{with}\quad
A_\infty\eqdef \left[
\begin{array}{ccc}
0\dots&\dots0&\overbrace{1\dots1}^{2^\ell\text{ times }}\\
0\dots0&\overbrace{2\dots2}^{2^{\ell-1}\text{ times }}&0\dots0\\
\vdots&\vdots&\vdots\\
2^{\ell}\;\;0\dots0&0\dots0&0\dots0
\end{array}
\right]\;.
$$
The general term is $2^\dj$ for $\dj\in\{0,\dots,\ell\}$. Relations~(\ref{eq:varianceZlim}) and~(\ref{eq:RjprimDef})  give that, for
$j'=\lowscale+\dj\in\{\lowscale,\dots,\lowscale+\ell\}$, 
$$
\sqrt{n2^{-\lowscale}} \PE[R_{j'}] = O\left((n2^{-\lowscale})^{-1/2}\right) \to 0\quad\text{as}\quad n\to\infty\; .
$$
Applying~(\ref{eq:CenteredZ}),~(\ref{eq:MatrixExprDWTwrtZ}), the two last displays and Slutsky's lemma, we
get~(\ref{eq:JointCentralLimitEmpVar}) with
$$
\AVvarJoint{d}{}=2^{-\ell}\; A_\infty\Gamma A_\infty^T
=2^{-\ell}\;  \left[2^{\dj+\dj'}\sum_{\dk=0}^{2^{\ell-\dj}-1}\sum_{\dk'=0}^{2^{\ell-\dj'}-1}
\Gamma_{2^{\ell-\dj}+\dk,2^{\ell-\dj'}+\dk'}\right]_{0\leq \dj,\dj'\leq \ell}\;,
$$
where the indices $(\dj,\dj')$ run from $(0,0)$ (top left corner) to $(\dj,\dj')=(\ell,\ell)$ (bottom right corner) and
$\Gamma_{i,i'}$ is defined by~(\ref{eq:GammaDef}) with $\vfonc_{i,\infty}$ and  $\vfonc_{i',\infty}$ defined
by~(\ref{eq:vfoncInftyDWTcase}) for $i,i'\in1,\dots,\dimV=2^{\ell+1}-1$.
To conclude the proof, it remains to check that the entries of $\AVvarJoint{d}{\dj,\dj'}$ as defined above are equal to those
given in~(\ref{eq:CovMatrix}). We shall do that for $\dj'\geq \dj$ since the alternative case is obtained by observing that 
$\AVvarJoint{d}{\dj,\dj'}=\AVvarJoint{d}{\dj',\dj}$.   
Replacing $\Gamma_{i,i'}$ and then $\vfonc_{i,\infty}$  and  $\vfonc_{i',\infty}$  by these expressions and denoting 
$$
\lambda_p\eqdef\lambda+2p\pi,\quad\lambda\in\Rset,\;p\in\Zset\;,
$$
we get, for $0\leq \dj,\dj'\leq \ell$, 
\begin{align}
\nonumber
  \AVvarJoint{d}{\dj,\dj'}&=2^{-\ell+\dj+\dj'}\sum_{\dk=0}^{2^{\ell-\dj}-1}\sum_{\dk'=0}^{2^{\ell-\dj'}-1}
4\pi\;(f^\ast(0))^2\\
\nonumber
&\times\int_{-\pi}^\pi2^{\dj+\dj'-2\ell+4\ell d}
\left|\sum_{p\in\Zset}\rme^{\rmi\lambda_p(\dk2^{\dj-\ell}-\dk'2^{\dj'-\ell})}
\;\left|\lambda_p\right|^{-2d}\overline{\hat{\psi}(2^{\dj-\ell}\lambda_p)}\hat{\psi}(2^{\dj'-\ell}\lambda_p)\right|^2\;\rmd\lambda \\
  \label{eq:Lambdad1}
&=(f^\ast(0))^2\;4\pi\;2^{2(\dj+\dj')+\ell(4d-3)}
\sum_{\dk=0}^{2^{\ell-\dj}-1} \int_{-\pi}^\pi G_{\dj,\dj',\dk}(\lambda) \;\rmd\lambda\;, 
\end{align}
where $G_{\dj,\dj',\dk}$ is a $(2\pi)$-periodic function defined by
$$
G_{\dj,\dj',\dk}(\lambda)\eqdef\sum_{\dk'=0}^{2^{\ell-\dj'}-1}
\left|\sum_{p\in\Zset}\rme^{\rmi\,\lambda_p(\dk2^{\dj-\ell}-\dk'2^{\dj'-\ell})}g_{\dj,\dj'}(2^{\dj'-\ell}\lambda_p)\right|^2
  \;,
$$
with
$$
g_{\dj,\dj'}(\lambda)\eqdef|2^{\ell-\dj'}\lambda|^{-2d}\overline{\hat{\psi}(2^{\dj-\dj'}\lambda)}\hat{\psi}(\lambda),\quad\lambda\in\Rset \; .
$$
Writing $p=2^{\ell-\dj'}q+r$ with $q\in\Zset$ and $r\in\{0,\dots,2^{\ell-\dj'}-1\}$ and transforming a sum over $p$ into a
sum over $q$ and $r$, we get
\begin{align*}
G_{\dj,\dj',\dk}(\lambda)
&=\sum_{\dk'=0}^{2^{\ell-\dj'}-1}\left|\sum_{r=0}^{2^{\ell-\dj'}-1}
\rme^{\rmi\,\lambda_r(\dk2^{\dj-\ell}-\dk'2^{\dj'-\ell})}\sum_{q\in\Zset}
\rme^{\rmi\,2^{\dj-\dj'}\dk2q\pi}g_{\dj,\dj'}(2^{\dj'-\ell}\lambda_r+2q\pi)\right|^2\\
&=\sum_{\dk'=0}^{2^{\ell-\dj'}-1}\left|\sum_{r=0}^{2^{\ell-\dj'}-1}
\rme^{-\rmi\,\lambda_r\dk'2^{\dj'-\ell}}h_{\dj,\dj',\dk}(2^{\dj'-\ell}\lambda_r)\right|^2\;,
\end{align*}
where
$$
h_{\dj,\dj',\dk}(\lambda)\eqdef\sum_{q\in\Zset}\rme^{\rmi\,2^{\dj-\dj'}\dk\lambda_q} g_{\dj,\dj'}(\lambda_q) \;.
$$
Hence
$$
G_{\dj,\dj',\dk}(\lambda)=\sum_{\dk'=0}^{2^{\ell-\dj'}-1}\sum_{r=0}^{2^{\ell-\dj'}-1}\sum_{r'=0}^{2^{\ell-\dj'}-1}
\rme^{-\rmi\,2\pi(r-r')\dk'2^{\dj'-\ell}} 
h_{\dj,\dj',\dk}(2^{\dj'-\ell}\lambda_r)\overline{h_{\dj,\dj',\dk}(2^{\dj'-\ell}\lambda_{r'})}\;.
$$
In the last display, observe that $\dk'$ only appear in the complex exponential argument. 
Moreover we have $\sum_{\dk'=0}^{2^{\ell-\dj'}-1}\rme^{-\rmi\,2\pi(r-r')\dk'2^{\dj'-\ell}}=0$ except for $r=r'$
in which case it equals $2^{\ell-\dj'}$. Hence, 
$$
G_{\dj,\dj',\dk}(\lambda)=2^{\ell-\dj'}\sum_{r=0}^{2^{\ell-\dj'}-1}
\left|h_{\dj,\dj',\dk}(2^{\dj'-\ell}\lambda_r)\right|^2 \; .
$$
Applying \cite[Lemma~1]{roueff:taqqu:2008a}
with $g=\left|h_{\dj,\dj',\dk}\right|^2$ and $\decim=2^{\ell-\dj'}$ gives
\begin{align*}
\int_{-\pi}^\pi G_{\dj,\dj',\dk}(\lambda) \;\rmd\lambda 
&=2^{\ell-\dj'}
\int_{-\pi}^\pi \sum_{r=0}^{2^{\ell-\dj'}-1} \left|h_{\dj,\dj',\dk}(2^{\dj'-\ell}\lambda_r)\right|^2 \;\rmd\lambda \\
&= 2^{2\ell-2\dj'} \; 
\int_{-\pi}^\pi\left|h_{\dj,\dj',\dk}(\lambda)\right|^2\;\rmd\lambda\;.
\end{align*}
Inserting this equality in~(\ref{eq:Lambdad1}) and using that $\AVvarJoint{d}{\dj,\dj'}=\AVvarJoint{d}{\dj',\dj}$, we get,
for all $0\leq \dj\leq \dj' \leq\ell$, 
\begin{multline*}
  \AVvarJoint{d}{\dj,\dj'}=(f^\ast(0))^2\;4\pi\;2^{2\dj+4d\dj'-\ell} \\
  \times \sum_{\dk=0}^{2^{\ell-\dj}-1}  \int_{-\pi}^\pi
\left|\sum_{q\in\Zset}|\lambda_q|^{-2d} \rme^{\rmi\,2^{-(\dj'-\dj)}\dk\lambda_q}
\overline{\hat{\psi}(2^{-(\dj'-\dj)}\lambda_q)}\hat{\psi}(\lambda_q) \right|^2\;\rmd\lambda\;.  
\end{multline*}
For $\dk\in\{0,\dots,2^{\ell-\dj}-1\}$, we write $\dk=\dk'+k2^{\dj'-\dj}$ with $\dk'\in\{0,\dots,2^{\dj'-\dj}-1\}$
and $k\in\{0,\dots,2^{\ell-\dj'}-1\}$ and transform the sum over $\dk$ into a sum over $\dk'$ and $k$. 
But since
$\exp\left\{\rmi2^{-(\dj'-\dj)}\dk\lambda_q\right\}=\exp\left\{\rmi2^{-(\dj'-\dj)}\dk'\lambda_q\right\}\exp\left\{\rmi k\lambda\right\}$,
    and $\sum_{k=0}^{2^{\ell-\dj'}-1}\left|\rme^{\rmi k\lambda}\right|^2=2^{\ell-\dj'}$, we obtain
\begin{multline*}
  \AVvarJoint{d}{\dj,\dj'}=(f^\ast(0))^2\;4\pi\;2^{4d\dj'+\dj}2^{-(\dj'-\dj)} \\
  \times \sum_{\dk'=0}^{2^{\dj'-\dj}-1}  \int_{-\pi}^\pi
\left|\sum_{q\in\Zset}|\lambda_q|^{-2d} \rme^{\rmi\,2^{-(\dj'-\dj)}\dk'\lambda_q}
\overline{\hat{\psi}(2^{-(\dj'-\dj)}\lambda_q)}\hat{\psi}(\lambda_q) \right|^2\;\rmd\lambda\;.  
\end{multline*}
Relation~(\ref{eq:CovMatrix}) finally follows by observing that the vector with entries
$$
2^{-(\dj'-\dj)/2}\left\{\rme^{\rmi\,2^{-(\dj'-\dj)}\dk'\lambda_q},\;\dk'=0,\dots,2^{\dj'-\dj}-1\right\}
$$ 
is precisely $\be_{\dj'-\dj}(\lambda_q)$.
\end{proof}

To obtain a result valid for an asymptotically infinite weighted sum of the empirical scalogram 
$\{\hvj{\lowscale+\dj}{n_{\lowscale+\dj}}-\vj{\lowscale+\dj}{d,f^\ast}{2},\;\dj\geq0\}$ as in
Theorem~\ref{thm:weakConvWeightedScalogramLimit} below, we need a bound for the 
covariance $\AVvarJoint{d}{\dj,\dj'}$ defined in~(\ref{eq:CovMatrix}) and a bound for the centered empirical scalogram. 
The two following results provide the bounds.

\begin{lemma}
  \label{lem:bDensBound}
Suppose that $\psi$ satisfies~\ref{item:Wreg}--\ref{item:MVM} and let $d\in(1/2-\alpha,M]$. Then, there exists $C$ only
depending on $d$ and $\psi$ such that, for all $\dj\geq0$,
$$
\int_{-\pi}^{\pi} \left|\bdensasymp{\dj}{\lambda}{d} \right|^2 \, \rmd\lambda\leq C\; 2^{\dj(1/2-2d)} \; .
$$
\end{lemma}
\begin{proof}
See Relation~(72) in~\cite{moulines-roueff-taqqu-2007c}. An alternative is to use that
$\AVvarJoint{d}{0,\dj}=\PCov\left(\hvjLIM{d}{0},\hvjLIM{d}{\dj}\right)$ 
and thus the Cauchy-Schwarz Inequality yields $|\AVvarJoint{d}{0,\dj}|^2\leq |\AVvarJoint{d}{0,0}||\AVvarJoint{d}{\dj,\dj}|$.  
Using~(\ref{eq:CovMatrix}), we get, setting $f^\ast(0)=1$
$$
4\pi2^{4d\dj}\int_{-\pi}^{\pi} \left|\bdensasymp{\dj}{\lambda}{d} \right|^2\, \rmd\lambda\leq 4\pi2^{(2d+1/2)\dj}\;\int_{-\pi}^{\pi}
\left|\bdensasymp{0}{\lambda}{d} \right|^2 \, \rmd\lambda\; .
$$
The results follows from the fact that $\left|\bdensasymp{0}{\lambda}{d} \right|$ is bounded for  $d\in(1/2-\alpha,M]$
under~\ref{item:Wreg}--\ref{item:MVM}, see Remark~1 in~\cite{moulines-roueff-taqqu-2007b}.  
\end{proof} 
\begin{lemma}\label{lem:BoundScalogram} 
Let $X$ be an $M(d)$ process with short-range spectral density $f^\ast$ 
and suppose that Assumption~\ref{assump:linear} holds.  
Assume that~\allWA\ hold with Condition~(\ref{eq:CondMalphadKsansBeta}) on $M$ and $\alpha$. Then, 
there exists a positive constant $C$ such that,  for all $n\geq1$ and $j\in\{0,1,\dots,\maxscale\}$, 
\begin{equation}
  \label{eq:VarScaloBothCases}
\PE\left[\left|\hvj{j}{n_j}-\vj{j}{d,f^\ast}{2}\right|\right]\leq C \left\{2^{(1/2+2d) j}\,n^{-1/2} + 2^{j(1-2\alpha)}\right\} \; .
\end{equation}
If moreover Condition~(\ref{eq:CondFstarBounded}) holds on $f^\ast$, one has
\begin{equation}
  \label{eq:VarScaloBoundedCase}
\PVar\left(\hvj{j}{n_j}\right)\leq C^2 \, 2^{(1+4d) j}\,n^{-1} \; .
\end{equation}
\end{lemma}
\begin{proof}
We use the same notations as in Theorem~\ref{theo:genericJointConvResult} to express $\hvj{j}{n_j}$
in terms of a decimated linear process, but since here only one scale needs to be considered, we take  
 $\lowscale=j=j'$ (hence $\dj=\ell=0$ and $i=1$). In this case~(\ref{eq:scaloLinearProc}) reads as
\begin{equation}
  \label{eq:CenteredScalogram}
2^{-2jd}(\hvj{j}{n_j}-\vj{j}{d,f^\ast}{2}) =
n_j^{-1}\sum_{k=0}^{n_j-1}\left\{Z_{1,j,k}^2-\PE[Z_{1,j,k}^2]\right\} \;,  
\end{equation}
where $Z_{1,j,k}=\sum_{t\in\Zset}\vseq_{1,j}(t)\xi_t$ with $\vseq_{1,j}(t) = 2^{-j
  d}\useq_{j}(t)$. If~(\ref{eq:CondFstarBounded}) holds, then $\vseq_{1,j}$ satisfies~(\ref{eq:unfiBoundvfoncLocale}) with $\varepsilon=\pi$
and~(\ref{eq:Limitvfonc}) (see the proof of Theorem~\ref{theo:genericJointConvResult}) and, by Lemmas~5
and~6 in~\cite{roueff:taqqu:2008a}, we get
\begin{equation}
  \label{eq:varscaloBonCas}
\sup_{j,n}\PVar\left(n_j^{-1/2}\sum_{k=0}^{n_j-1}Z_{1,j,k}^2\right) < \infty \; .
\end{equation}  
Since $n_j\asymp n2^{-j}$ for $j\in\{0,1,\dots,\maxscale\}$,~(\ref{eq:VarScaloBoundedCase}) follows.

If~(\ref{eq:CondFstarBounded}) does not hold,  
$\vseq_{1,j}$ satisfies~(\ref{eq:unfiBoundvfoncLocale}) for some $\varepsilon>0$ which may no longer be taken equal to $\pi$
(as a consequence of~(\ref{eq:localBoundvfoncCLTwav}) in the proof of Theorem~\ref{theo:genericJointConvResult}) and,
applying ~\cite[Proposition~4]{roueff:taqqu:2008a} with $\lambda_{1,j}=0$, 
we get
$$
2^{-2jd}(\hvj{j}{n_j}-\vj{j}{d,f^\ast}{2}) =n_j^{-1/2}\left[n_j^{-1/2}\sum_{k=0}^{n_j-1} \{\widehat{Z}_{1,j,k}^{2} -\PE[\widehat{Z}_{1,j,k}^{2}]\}
+R_j\right] \;,
$$
where $\widehat{Z}_{1,j,k}$  satisfies~(\ref{eq:unfiBoundvfoncLocale}) with $\varepsilon=\pi$ and~(\ref{eq:Limitvfonc}) and
hence~(\ref{eq:varscaloBonCas}) and $R_j$ satisfies, for some positive constant $C$ not depending on $j$, 
\begin{equation}\label{eq:RjBound}
\PE\left[\left|R_j\right|\right]\leq C\;
\left[n_j^{1/2}I_j+I_j^{1/2}\right]\;,
\end{equation}
where
\begin{equation}\label{eq:Ijdef}
I_j\eqdef\int_{0}^\pi\1(|\lambda-\lambda_{1,\infty}|>\varepsilon)\;\left|\vfonc_{1,j}(\lambda)\right|^2\;\rmd\lambda \;.
\end{equation}
Since  $f^\ast$ is always integrable away of the origin, the bound~(\ref{eq:localBoundvfoncCLTwav2}) implies (recall that
here $\lowscale=j$)
$$
\int_{\varepsilon}^\pi|\vfonc_{1,j}(\lambda)|^2\;\rmd\lambda \leq C\; 2^{(1-2\alpha-2d)j} \; .
$$
Hence, 
we obtain, for some constant $C$ not depending on $j$ nor $n$, 
\begin{align*}
\PE\left[2^{-2jd}|\hvj{j}{n_j}-\vj{j}{d,f^\ast}{2}| \right]&\leq C\;n_j^{-1/2}\left[1+n_j^{1/2}2^{(1-2\alpha-2d)j} +
  2^{(1-2\alpha-2d)j/2} \right] \\
&\leq  C\;n_j^{-1/2}\left[2+n_j^{1/2}2^{(1-2\alpha-2d)j} \right] \;, 
\end{align*}
where we used $1-2\alpha-2d<0$ in Condition~(\ref{eq:CondMalphadKsansBeta}). Since $n_j\asymp n2^{-j}$
Relation~(\ref{eq:VarScaloBothCases}) follows.
\end{proof}

We now prove the main result of this section.

\begin{theorem}\label{thm:weakConvWeightedScalogramLimit}
Let $\{w_{\dj}(n),n,\dj\geq0\}$ be an array of real numbers such that 
$w_{\dj}(n)\to w_{\dj}$ for all $\dj\geq0$ as $n\to\infty$ and  
\begin{equation}
  \label{eq:poidsDomination}
  \lim_{\ell\to\infty}\sum_{\dj>\ell}\sup_{n\geq0}|w_{\dj}(n)|2^{(1/2+2d)\dj} = 0 \;.
\end{equation}
Then, under the assumptions of Theorem~\ref{theo:genericJointConvResult}, as $n\to\infty$,
\begin{equation}
\label{eq:WeightedScalogramLimit}
\sqrt{n2^{-\lowscale}}2^{-2\lowscale d}\sum_{\dj=0}^{\maxscale-\lowscale}w_{\dj}(n)
\{\hvj{\lowscale+\dj}{n_{\lowscale+\dj}}-\vj{\lowscale+\dj}{d,f^\ast}{2}\} \cl
 \calN\left(0, \sum_{\dj,\dj'\geq0}w_{\dj}\AVvarJoint{d}{\dj,\dj'}w_{\dj'} \right) \; ,
\end{equation}
where $\maxscale$ is defined in~(\ref{eq:defMaxscale}).
\end{theorem}
\begin{proof}
We denote, for all $\ell\geq0$,
$$
S_{n,\ell} = \sqrt{n2^{-\lowscale}}2^{-2\lowscale d}\sum_{\dj=0}^{\ell}w_{\dj}(n)
\{\hvj{\lowscale+\dj}{n_{\lowscale+\dj}}-\vj{\lowscale+\dj}{d,f^\ast}{2}\} 
$$
and
$$
\tilde{S}_{n,\ell} = \sqrt{n2^{-\lowscale}}2^{-2\lowscale d}\sum_{\dj=0}^{\ell}w_{\dj}
\{\hvj{\lowscale+\dj}{n_{\lowscale+\dj}}-\vj{\lowscale+\dj}{d,f^\ast}{2}\} \;.
$$
Theorem~\ref{theo:genericJointConvResult} then gives that, for any $\ell\geq0$, as $n\to\infty$,
$$
\tilde{S}_{n,\ell} \cl
 \calN\left(0, \sum_{0\leq\dj,\dj'\leq\ell}w_{\dj}\AVvarJoint{d}{\dj,\dj'}w_{\dj'} \right) \; .
$$
Note that~(\ref{eq:poidsDomination}) implies $\sum_{\dj>\ell}|w_{\dj}|2^{(1/2+2d)\dj}\to0$ as $\ell\to\infty$, hence,
using Lemma~\ref{lem:bDensBound} and~(\ref{eq:CovMatrix}), we have
\begin{align*}
\sum_{\ell<\dj,\dj'}\left|w_{\dj}\AVvarJoint{d}{\dj,\dj'}w_{\dj}\right|
&\leq
C\;\sum_{\ell<\dj\leq\dj'}\left|w_{\dj}w_{\dj}\right|2^{4d\dj'+\dj}2^{(\dj'-\dj)(1/2-2d)}\\
&\leq C\; \sum_{\ell<\dj}|w_{\dj}|2^{(1/2+2d)\dj} \times \sum_{\ell<\dj'}|w_{\dj'}|2^{(1/2+2d)\dj'}\\
&\to0\quad\text{as $\ell\to\infty$.}  
\end{align*} 
The left-hand side of~(\ref{eq:WeightedScalogramLimit}) is $S_{n,J-\lowscale}$. We decompose it as 
$$
S_{n,\maxscale-\lowscale}=\left[S_{n,\maxscale-\lowscale}-S_{n,\ell}\right]+\left[S_{n,\ell}-\tilde{S}_{n,\ell}\right]+\tilde{S}_{n,\ell}
$$
From the last 3 displays and applying~\cite[Theorem~3.2]{billingsley:1999},  it is sufficient to prove that
\begin{equation}
  \label{eq:cltweightedRemainder}
  \lim_{\ell\to\infty}\limsup_{n\to\infty}\PE\left[\left|S_{n,\maxscale-\lowscale}-S_{n,\ell}\right|+\left|S_{n,\ell}-\tilde{S}_{n,\ell}\right|\right] =0\;.  
\end{equation}
To obtain this limit, we need to separate the case where Condition~(\ref{eq:CondFstarBounded}) holds from the one where it
is replaced by Condition~(\ref{eq:CondFstarUNBounded}). Under  Condition~(\ref{eq:CondFstarBounded}), we 
apply~(\ref{eq:VarScaloBoundedCase}); under~(\ref{eq:CondFstarUNBounded}), we apply~(\ref{eq:VarScaloBothCases}).
Let us for instance check the second case (the first one is similar, although simpler). The bound~(\ref{eq:VarScaloBothCases})
implies 
$$
\PE\left[\left|S_{n,\ell}-\tilde{S}_{n,\ell}\right|\right]\leq C\,\sum_{\dj=0}^\ell |w_{\dj}-w_{\dj}(n)| 
(2^{(1/2+2d)\dj}+2^{\lowscale(1-2\alpha-2d)}\sqrt{n2^{-\lowscale}}\;2^{(1-2\alpha)\dj})\;,
$$
which, using $w_{\dj}\to w_{\dj}(n)$ and~(\ref{eq:CondFstarUNBounded}),  tends to 0 as $n\to\infty$ for all $\ell\geq0$, and 
\begin{multline*}
\PE\left[\left|S_{n,\maxscale-\lowscale}-S_{n,\ell}\right|\right]\\
\leq 
C\,\left[\sum_{\dj>\ell}|w_{\dj}(n)|2^{(1/2+2d)\dj}+
2^{\lowscale(1-2\alpha-2d)}\sqrt{n2^{-\lowscale}}\;\sum_{\dj>\ell}|w_{\dj}(n)|2^{\dj(1-2\alpha)} \right] \\
\leq C\,\left[1+2^{\lowscale(1-2\alpha-2d)}\sqrt{n2^{-\lowscale}} \right] \sum_{\dj>\ell}|w_{\dj}(n)|2^{(1/2+2d)\dj}\; ,
\end{multline*}
(since $1-2\alpha<2d$ in Condition~(\ref{eq:CondMalphadKsansBeta})) which tends to 0 as $n\to\infty$ followed by
$\ell\to\infty$ by~(\ref{eq:CondFstarUNBounded})  and~(\ref{eq:poidsDomination}). This
yields~(\ref{eq:cltweightedRemainder}), which achieves the proof.  
\end{proof}

\section{The log-regression estimation of the memory parameter} \label{sec:log-regr-estim}

The wavelet-based regression estimator of the memory parameter $d$ involves regressing the scale spectrum
estimator $\hvj{j}{n_j}$, defined in~(\ref{eq:hvj}), with respect to the scale index $j$.
More precisely, an estimator of the memory parameter $d$ is obtained by
regressing the logarithm of the empirical variance $\log (\hvj{i}{n_i})$  for a finite number of scale indices
$j \in \{\lowscale, \dots, \lowscale+\ell\}$
where $\lowscale=\lowscale(n)\geq0$ is the lower scale and $1+\ell\geq2$ is the number of scales used in the regression.
For a sample size equal to $n$, this estimator is well defined for $\lowscale$ and $\ell$ such that $\ell\geq1$
and
\begin{equation}
\label{eq:def:Jn}
\lowscale+\ell \leq [\log_2 (n-\L+1)-\log_2(\L)]\;,
\end{equation}
where the right-hand side of this inequality is the maximal index $j$ such that $n_j\geq1$. The regression estimator can be
expressed formally as 
\begin{equation}
\label{eq:definition:estimator:regression}
\hat{d}_n(\lowscale,\regressweights) \eqdef \sum_{j=\lowscale}^{\lowscale+\ell} w_{j-\lowscale} \log \left( \hvj{j}{n_j} \right) \eqsp ,
\end{equation}
where the vector $\regressweights \eqdef[w_0,\dots,w_{\ell}]^T$  of weights satisfies
\begin{equation}
\label{eq:propertyw}
\sum_{i=0}^{\ell} w_{i}  = 0\quad\text{and}\quad 2 \log(2) \sum_{i=0}^{\ell} i w_{i}  = 1 \eqsp.
\end{equation}
One may choose, for example, $\regressweights$ corresponding to the weighted least-squares regression vector, defined by
\[
\regressweights= D B (B^TDB)^{-1} \mathbf{b}  \eqsp,
\]
where
$B \eqdef \left[\begin{matrix}
1 & 1 & \dots & 1 \\
0 & 2 & \dots & \ell
\end{matrix}\right]^T$ is the so-called design matrix, $D$ is a definite positive matrix and
\begin{equation}\label{eq:bfbDef}
\mathbf{b}\eqdef [0 \,\, (2\log(2))^{-1}]^T.
\end{equation}
Ordinary least square regression corresponds to the case where $D$ is the identity matrix.

In \cite{moulines-roueff-taqqu-2007a}, the process $X$ was assumed Gaussian and a bound for the mean square error and an
asymptotic equivalent to the variance of $\hat{d}_n(\lowscale,\regressweights)$ were obtained.
The asymptotic normality is established in~\cite{moulines-roueff-taqqu-2007b}, also under the Gaussian assumption. 
Here we show that the asymptotic normality holds under the weaker linear assumption. 

\begin{theorem}
\label{theo:asympvariance}
Let $X$ be an $M(d)$ process with short-range spectral density $f^\ast$ 
and suppose that Assumptions~\ref{H:Lbeta} and~\ref{assump:linear} hold.  
Under \allWA with
\begin{equation}
  \label{eq:CondMalphadKavecBeta}
   (1+\beta)/2-\alpha< d\leq M\quad\text{and}\quad K\leq M \;,
\end{equation}
 if, as $n\to\infty$, $\lowscale(n)$ is such that  
\begin{equation}
\label{eq:mAndnconditionForCLT}
 (n2^{-\lowscale(n)})^{-1} + n2^{-(1+2\beta)\lowscale(n)}\to0\; ,
\end{equation}
then one has the following central limit:
\begin{equation}
\label{eq:AbryVeitchEstimatorCLTGaussian}
\sqrt{n2^{-\lowscale(n)}} \left(\hat{d}_n(\lowscale,\regressweights) - d \right) \cl \calN\left(0,\regressweights^T \AVvar{d}{}
  \regressweights\right) \eqsp, 
\end{equation}
where 
\begin{equation}\label{eq:Kpsi}
\Kvar[\psi]{d}\eqdef \int_{-\infty}^{\infty} |\xi|^{-2d}\,|\hat\psi(\xi)|^2\,\rmd\xi \eqsp,
\end{equation}
and $\AVvar[\psi]{d}{}$ is the $(1+\ell)\times(1+\ell)$ matrix defined as
\begin{align}
\label{eq:definitionAVvar}
\AVvar[\psi]{d}{i,j} \eqdef \frac{4 \pi 2^{2d |j-i| } 2^{i \wedge j}}{\Kvar[\psi]{d}^2} \int_{-\pi}^{\pi} \left|
  \bdensasymp[\psi]{|j-i|}{\lambda}{d} \right|^2 \, \rmd\lambda  && 0 \leq i, j \leq \ell \eqsp ,
\end{align}
\end{theorem}
\begin{proof}
The only assumptions of Theorem~\ref{theo:genericJointConvResult} that are not included in our set of assumptions here are \eqref{eq:CondFstarBounded} and
\eqref{eq:CondFstarUNBounded}. The former is verified if $\varepsilon=\pi$ in Assumption~\ref{H:Lbeta}. If $\varepsilon<\pi$, 
\eqref{eq:CondFstarUNBounded} holds as a consequence of~(\ref{eq:CondMalphadKavecBeta}) and~\eqref{eq:mAndnconditionForCLT}. Hence 
Theorem~\ref{theo:genericJointConvResult} applies (note that~(\ref{eq:CondMalphadKavecBeta})
implies~(\ref{eq:CondMalphadKsansBeta}) since $\beta>0$). 
Applying~\cite[Theorem~1]{moulines-roueff-taqqu-2007c}, under \allWA, we have the following approximation:
\begin{equation}\label{eq:VjApprox}
\left| \vj{j}{f}2 - f^\ast(0) \, \Kvar[\psi]{d} \, 2^{2jd} \right| \leq C\, f^\ast(0) \, L \, 2^{(2d-\beta)j}
\end{equation}
where $\vj{j}{f}2$ is defined in~(\ref{eq:definitionvj}) and $\Kvar[\psi]{d}$ in~(\ref{eq:Kpsi}).
This, with~Theorem~\ref{theo:genericJointConvResult}, ~\cite[Relation~(39)]{moulines-roueff-taqqu-2007b}
and~\cite[Proposition~3]{moulines-roueff-taqqu-2007b}, gives the result. 
\end{proof}
\begin{remark}
  The centering in~(\ref{eq:ScalogramLimit}) in Theorem~\ref{theo:genericJointConvResult} involved the expected value whereas
  the centering in~(\ref{eq:AbryVeitchEstimatorCLTGaussian}) in Theorem~\ref{theo:asympvariance} does not. In order to deal
  with the corresponding bias, Condition~(\ref{eq:CondMalphadKsansBeta}) is strengthened by~(\ref{eq:CondMalphadKavecBeta}).
\end{remark}

\section{The wavelet Whittle estimation of the memory parameter}\label{sec:wavel-whittle-estim}

We now consider the semi-parametric estimator introduced in~\cite{moulines-roueff-taqqu-2007c}.
As the log-regression estimation, this estimator is also based on the scalogram but is defined as the maximizer of a
Whittle type contrast function (see~\cite[Eq.~(20)]{moulines-roueff-taqqu-2007c}), 
$$
\tilde{d}_n(\lowscale,\upscale)\eqdef
\argmin_{d'\in\Rset} \left[\log\left(\sum_{j=\lowscale}^{\upscale}2^{-2d'j}n_j  \hvj{j}{n_j} \right) 
+2 d' \log(2) \jmean \right] \quad\text{with}\quad\jmean\eqdef \frac{\sum_{j=\lowscale}^{\upscale}j\;n_j}{\sum_{j=\lowscale}^{\upscale}n_j} \; .
$$
$\tilde{d}_n$, which involves the scales $\lowscale\leq j\leq \upscale$,
is a wavelet analog of the local Whittle Fourier estimator studied in \cite{robinson:1995g} (often referred to as
\emph{semiparametric Gaussian estimator}) and is therefore called the \emph{local Whittle wavelet} estimator. To prove the
asymptotic normality of $\tilde{d}_n$, we will use Theorem~\ref{thm:weakConvWeightedScalogramLimit}.

We denote, for all integer $\k\geq1$,
\begin{gather}
\label{eq:eta-k}
\eta_\k\eqdef   \sum_{j=0}^\k  j \frac{2^{-j}}{2-2^{-\k}}
\quad\text{and}\quad
\kappa_\k\eqdef\sum_{j=0}^{\k} (j-\eta_{\k})^2 \frac{2^{-j}}{2-2^{-\k}}\;, \\
\nonumber
\AsympVarWWE[\psi]{d,\k} \eqdef \frac{\pi}{(2-2^{-\k})\kappa_\k(\log(2)\Kvar[\psi]{d})^2} \times \hspace{6cm} \\
\label{eq:varLimite-k}
\hspace{1cm}\left\{\intbdens[\psi]{0}{d} +
 \frac2{\kappa_\k}
\sum_{\dj=1}^\k \intbdens[\psi]{\dj}{d} \, 2^{(2d-1) \dj}\,\sum_{i=0}^{\k-\dj}\frac{2^{-i}}{2-2^{-\k}}(i-\eta_\k)(i+\dj-\eta_\k) \right\}\eqsp , \\
\label{eq:varLimite}
\AsympVarWWE[\psi]{d,\infty} \eqdef
\frac{\pi}{[2\log(2)\Kvar[\psi]{d}]^2} \left\{\intbdens[\psi]{0}{d} +
2\sum_{\dj=1}^\infty \intbdens[\psi]{\dj}{d} \, 2^{(2d-1) \dj} \right\} \eqsp ,
\end{gather}
where $\Kvar[\psi]{d}$ is defined in \eqref{eq:Kpsi}.

\begin{theorem}
\label{theo:CLTwhittle}
Let $X$ be an $M(d)$ process with short-range spectral density $f^\ast$ 
and suppose that Assumptions~\ref{H:Lbeta} and~\ref{assump:linear} hold.  
Under \allWA\ with Condition~(\ref{eq:CondMalphadKavecBeta}) on $\alpha$ and $M$,
if, as $n\to\infty$, the lower scale $\lowscale(n)$ is such that  
\begin{equation}
\label{eq:mAndnconditionForCLTWhittle}
\lowscale(n) (n2^{-\lowscale(n)})^{-1/8}+ n2^{-(1+2\beta)\lowscale(n)}\to0\; ,
\end{equation}
and the upper scale $\upscale(n)$ is such that
$$
\upscale(n)-\lowscale(n)\to \ell \in\{1,2\dots,\infty\},
$$
then one has the following central limit:
\begin{equation}
\label{eq:WhittleEstimatorCLTGaussian}
\sqrt{n2^{-\lowscale(n)}} \left(\tilde{d}_n(\lowscale(n),\upscale(n)) - d \right) \cl \calN\left(0,\AsympVarWWE[\psi]{d,\ell}\right) \eqsp,
\end{equation}
where $\AsympVarWWE[\psi]{d,\ell}$ for $\ell<\infty$ and $\ell=\infty$ are defined in~(\ref{eq:varLimite-k})
and~(\ref{eq:varLimite}) respectively.
\end{theorem}
\begin{remark}
Condition~(\ref{eq:mAndnconditionForCLTWhittle}) is similar to~(\ref{eq:mAndnconditionForCLT}) but
$n2^{-\lowscale(n)}\to\infty$ is replaced by the stronger condition $\lowscale(n) (n2^{-\lowscale(n)})^{-1/8}\to0$, which holds for example, if
$n2^{-\lowscale(n)}/ n^\gamma$ has a positive limit for some $\gamma>0$ as usually verified.
\end{remark}
\begin{proof}
Assume $f^\ast(0)=1$ without loss of generality.
  The proof is the same as that of~\cite[Theorem~5]{moulines-roueff-taqqu-2007c} until Eq.~(66),
  \begin{equation}
    \label{eq:66-MRT2007c}
    (n2^{-\lowscale})^{1/2}\,(\hd_n-d) =
    \frac{(n2^{-\lowscale})^{-1/2}\,\Sclt_n}{2\,\log(2)\,\Kvar[\psi]{d}
\,(2-2^{-(\upscale-\lowscale)})
      \,\kappa_{\upscale-\lowscale}} \,(1+o_{\prob}(1))\eqsp,     
\end{equation}
where
$$
\Sclt_n \eqdef\sum_{j=\lowscale}^\upscale [j-\jmean] \, 2^{-2j d }\, n_j  \hvj{j}{n_j}  \eqsp .
$$  
Thus, using $\sum_{j=\lowscale}^\upscale [j-\jmean]\, n_j =0$ and~(\ref{eq:definitionvj}), we get
$$
\PE \left[ \Sclt_n \right]= 
\sum_{j=\lowscale}^\upscale [j-\jmean] \, n_j \,\left(2^{-2j d }\,\vj{j}{f}{2}-f^\ast(0) \, \Kvar[\psi]{d}\right) = O( n2^{-(1+\beta) \lowscale})
= o\left ( (n2^{-\lowscale})^{1/2} \right)\eqsp,
$$
where the $O$-term follows from~(\ref{eq:VjApprox}), the fact that $\lowscale< \jmean<\lowscale+1$
(see~\cite[Eq~(61)]{moulines-roueff-taqqu-2007c}) and $n_j\leq n2^{-j}$ and the $o$-term follows from~(\ref{eq:mAndnconditionForCLTWhittle}).

Hence it only remains to establish a CLT for $\Sclt_n$ similar to that of \cite[Proposition~10]{moulines-roueff-taqqu-2007c} but
under the assumptions of Theorem~\ref{theo:CLTwhittle}. This is obtained by observing that
$$
(n2^{-\lowscale})^{-1/2}\,\Sclt_n=(n2^{-\lowscale})^{1/2}\,2^{-2\lowscale d}\sum_{\dj=0}^{\upscale-\lowscale} w_{\dj}(n)  \hvj{j}{n_j}  
$$
with $\dj=j-\lowscale$ and
$$
w_{\dj}(n)\eqdef [\dj- (\jmean-\lowscale)] \, 2^{-2\dj d }\, \frac{n_{\lowscale+\dj}}{n2^{-\lowscale}},\quad\dj\in
\{0,\dots,\upscale-\lowscale\}\;, 
$$
which satisfies
$$
\sup_{n}|w_{\dj}(n)|\leq C \; \dj 2^{-(2 d+1)\dj },\quad \dj\in\{0,\dots,\upscale-\lowscale\} \;,
$$
and by applying Theorem~\ref{thm:weakConvWeightedScalogramLimit}.
\end{proof}
\appendix

\section{Wavelet coefficients linear filters}

Assumption~\ref{item:Wreg} implies that $\hat{\phi}$ and $\hat{\psi}$ are everywhere infinitely differentiable.
When~\ref{item:Wreg}  holds, Assumptions~\ref{item:MVM} and~\ref{item:MIM} can be expressed in different ways. 
\ref{item:MVM}
is equivalent to asserting that the first $M-1$ derivative of $\hat{\psi}$ vanish at
the origin and hence
\begin{equation}
\label{eq:MVM}
|\hat{\psi}(\lambda)|=O(|\lambda|^{M})\quad\text{as}\quad\lambda\to0.
\end{equation}
And, by \cite[Theorem~2.8.1, Page~90]{cohen:2003},~\ref{item:MIM} is equivalent to
\begin{equation}
\label{eq:MIM}
\sup_{k\neq0} |\hat{\phi}(\lambda+2k\pi)|=O(|\lambda|^{M})\quad\text{as}\quad\lambda\to0.
\end{equation}

Many authors suppose that the $\psi_{j,k}$ are orthogonal and even that they are generated by a multiresolution analysis (MRA).
Assumptions~\ref{item:Wreg}--\ref{item:MIM} in Section~\ref{sec:quadr-forms-wavel} are satisfied in these cases, $\phi$ being the scaling
function and $\psi$ is the associated wavelet.
In this paper, however, we do not assume that wavelets are orthonormal nor
that they are associated to a multiresolution analysis. We may therefore work with other convenient choices for $\phi$ and $\psi$
as long as~\ref{item:Wreg}--\ref{item:MIM} are satisfied.  A simple example is to set, for some positive integer $N$,
$$
\phi(x) \eqdef \1_{[0,1]}^{\otimes N}(x)\quad\text{and}\quad\psi(x) \eqdef C_N\;\frac{\rmd^N}{\rmd x^N} \1_{[0,1]}^{\otimes 2N}(x),
$$
where $\1_A$ is the indicator function of the set $A$, $f^{\otimes N}$ denotes the $N$-th  self-convolution of
a function $f$ and $C_N$ is a normalizing constant such that $\int_{-\infty}^\infty\psi^2(x)\rmd x=1$. It follows that
$$
|\hat\phi(\xi)|=|2\sin(\xi/2)/\xi|^{N}\quad\text{and}\quad|\hat\psi(\xi)|=C_N\;|\xi|^{N}|2\sin(\xi/2)/\xi|^{2N}.
$$
Using~(\ref{eq:MVM}) and~(\ref{eq:MIM}), one easily checks that~\allWA\ are satisfied with $M$ and $\alpha$ equal to $N$.
Of course the  family of functions $\{\psi_{j,k}\}$ are not orthonormal for this choice of the wavelet function $\psi$ (and the function
$\phi$ is not associated to a MRA).

\bibliography{lrd}
\bibliographystyle{plain}

\end{document}